\documentclass{article}

\usepackage[margin=1in,a4paper]{geometry}
\usepackage{amssymb,amsmath,,wasysym}
\usepackage{amsthm} 
\usepackage{xcolor}
\usepackage{float}
\usepackage[colorlinks,linkcolor=blue,citecolor=blue,pagebackref,hypertexnames=false, breaklinks]{hyperref}
\usepackage{circuitikz}

\theoremstyle{plain}
\newtheorem{thm}{Theorem}[section]
\newtheorem{lem}[thm]{Lemma}
\newtheorem{cor}[thm]{Corollary}

\newtheorem{conj}[thm]{Conjecture}
\newtheorem{example}[thm]{Example}

\newenvironment{Example}{\begin{example}\rm}{\end{example}}

\theoremstyle{definition}
\newtheorem{defn}[thm]{Definition}

\theoremstyle{remark}
\newtheorem{remark}[thm]{Remark}
\newcommand{\bremark}{\begin{remark} \em}
\newcommand{\eremark}{\end{remark} }
\usepackage{color}

\hbadness=\maxdimen

\usepackage[textwidth=3.2cm]{todonotes}

\begin{document}


\title{Oscillations of BV measures on nested fractals}

\author{Patricia Alonso Ruiz\footnote{Partly supported by the NSF grant DMS 1951577}, Fabrice Baudoin\footnote{Partly supported by the NSF grant DMS~1901315.}}

\maketitle

\begin{abstract}
Motivated by recent developments in the theory of bounded variation functions on unbounded nested fractals, 
this paper studies the exact asymptotics of functionals related to the total variation measure associated with unions of $n$-cells. The oscillatory behavior observed implies the non-uniqueness of BV measures in this setting.

\vspace*{2em}
\end{abstract}

\tableofcontents
 
\section{Introduction}
Functions of bounded variation (BV) and their bounded variation measures are tightly connected to the geometry of the underlying space. Already in the 1920s, Caccioppoli characterized the perimeter measure of Euclidean sets as the BV measure associated with the corresponding indicator functions, c.f.~\cite{Cac52}. He observed that the perimeter of any measurable $E\subseteq\mathbb{R}^d$ coincides with the total variation norm of $\mathbf{1}_E$, i.e.
\begin{equation}\label{E:def_perim_Eucl}
\|D\mathbf{1}_E\|(\mathbb{R}^d):=\sup\Big\{\int_E{\rm div}\phi\,dx\colon \phi\in C_c(\mathbb{R}^d),\|\phi\|_\infty\leq 1\Big\}={\rm Perimeter}\,(E).
\end{equation}
This observation led to a general definition of sets with finite perimeter as being those for which the left hand side of~\eqref{E:def_perim_Eucl} is finite. Sets of finite perimeter are thus also referred to as Caccioppoli sets.

\medskip

The above characterization provided a natural way to extend the concept of (finite) perimeter beyond the Euclidean setting that would consist in finding a suitable analogue of ${\rm div}\,\phi$. An especially successful approach was found in the concept of weak upper gradients, developed in seminal works by Koskela~\cite{Kos00} and Shanmugalingam~\cite{Sha00}. This theory has led to many developments in the understanding of the interconnection between analytic and geometric properties of a variety of metric measure spaces admitting enough rectifiable curves. We refer the reader to~\cite{nagesbook} and references therein. 

\medskip

Another successful approach to characterize the BV measure~\eqref{E:def_perim_Eucl} in the context of metric measure spaces, that bypasses the above mentioned rectifiability requirements, relies on early work due to Korevaar-Schoen~\cite{KS97}. In this case, $\|D\mathbf{1}_E\|(\mathbb{R}^d)$ is comparable to  
\begin{equation}\label{E:BV_KS_Eucl}
\liminf_{r\to 0^+}\frac{1}{\sqrt{r}}\int_{\mathbb{R}^d}\int_{B(x,r)}\frac{|\mathbf{1}_E(x)-\mathbf{1}_E(y)|}{r^{d}}dy\,dx.
\end{equation}
The latter expression, with the same scaling $\sqrt{r}$, extends to Riemannian manifolds and more generally to Dirichlet spaces with Gaussian heat kernel estimates, see~\cite{ARB21b,BV2,MMS16}. When the heat kernel satisfies sub-Gaussian estimates, as the nested fractals $(X,d,\mu)$ considered in the present paper do, general BV functions and their total variation measures were introduced in~\cite[Section 4]{BV3}. These measures are comparable to
\begin{equation}\label{E:BV_KS_fractal} 
\nu_f(X):=\liminf_{r\to 0^+}\frac{1}{r^{\alpha_1 d_w}}\int_{X}\int_{B(x,r)}\frac{|f(x)-f(y)|}{r^{d_h}}d\mu(y)\,d\mu(x),
\end{equation}
where $f\in L^1(X,\mu)$, $d_h$ is the Hausdorff dimension and $d_w$ the so-called walk dimension of the space. $\alpha_1$ is a suitable critical exponent that guarantees a finite non-trivial $\liminf$ in~\eqref{E:BV_KS_fractal} for sufficiently many functions. In the case of Gaussian heat kernel estimates, this exponent is known to be $\alpha_1=1/2$, c.f.~\cite[Section 4.2]{BV2}, while for unbounded nested fractals $\alpha_1=d_h/d_w$, see Section~\ref{S:KS_BV} and~\cite[Section 4]{BV3} for details. Intuitively and loosely speaking, the parameter $d_h-\alpha_1 d_w$ can be interpreted as a minimal dimension of the measure theoretical boundary of open sets, see ~\cite[Section 2.4]{BV3}.
One advantage of~\eqref{E:BV_KS_fractal} is that it allows to perform rather explicit computations in specific examples, a feature that will be key in proving the non-uniqueness of BV measures in Section~\ref{SS:KS_non-limit} when $f=\mathbf{1}_E$ and $E\subseteq X$ is expressible as a finite union of cells. 
To that end, the BV measure~\eqref{E:BV_KS_fractal} will be 
expressed in terms of the functional
\begin{equation}\label{E:def_KS_functional}
\widetilde{\mathcal{M}}_f(r):=\frac{1}{r^{d_h}}\int_{X}\int_{B(x,r)}|f(x)-f(y)|\,d\mu(y)\,d\mu(x).
\end{equation}


A third approach to characterize the perimeter measure in the context of nested fractals and other Dirichlet spaces with sub-Gaussian heat kernel estimates makes use of the intrinsic diffusion process associated with the underlying space. Defining the functional
\begin{equation}\label{E:def_HS_functional}
\mathcal{M}_f (t) :=\int_{X} \int_{X} |f (x) -f(y)| p_t(x,y) \, d\mu (x) d\mu(y)
\end{equation}
for any $f\in L^1(X,\mu)$, it was proved in~\cite[Theorem 4.2]{BV3} that BV functions are characterized as those integrable for which $\sup_{r>0}r^{-\alpha_1d_w}\widetilde{\mathcal{M}}_f(r)$ is finite, and for these it holds that
\begin{equation}\label{E:BV3_main_thm}
\limsup_{t\to 0^+}\frac{1}{t^{d_h/d_w}}\mathcal{M}_f(t)\apprle \limsup_{r\to 0^+}\frac{1}{r^{d_h}}\widetilde{\mathcal{M}}_f(r)\apprle\liminf_{r\to 0^+}\frac{1}{r^{d_h}}\widetilde{\mathcal{M}}_f(r)\apprle \liminf_{t\to 0^+}\frac{1}{t^{d_h/d_w}}\mathcal{M}_f(t).
\end{equation}
The first and third inequality above follow from~\cite[Lemma 4.13]{BV3} and the second from~\cite[Theorem 4.9]{BV3}.
The heat semigroup based characterization of BV functions combines ideas going back to de Giorgi~\cite{deG54} and Ledoux~\cite{Led94}, which had been used to prove the analogue result in the Riemannian manifold setting with $\alpha_1=1/2$, see~\cite{MPPP07}. In fact, the exact characterization of the BV measures, or total variation of functions, on a Riemannian manifold $\mathbb{M}$ through the heat semigroup has only been proved recently in~\cite{ARB21b}. Namely, for any BV function $f$, both $\limsup$ and $\liminf$ in~\eqref{E:BV3_main_thm} coincide and moreover
\begin{equation}\label{E:RiemBV_main_thm}
\|Df\|(\mathbb{M})=\lim_{t\to 0^+}\frac{\sqrt{\pi}}{2\sqrt{t}}\mathcal{M}_f(t).
\end{equation}
The present paper shows that nested fractals behave quite differently. In particular, we find unbounded nested fractals that present an oscillatory behavior of the BV measure of certain indicator functions.
This fact, recorded in Theorem~\ref{T:liminf_Ucells_intro}, is therefore a refinement of the estimates~\eqref{E:BV3_main_thm} when $f$ is the indicator function of a finite union of cells like those illustrated in Figure~\ref{F:U_cells_example_intro}. Such functions were proved to be BV in~\cite[Theorem 5.1]{BV3}. For such a union of $n$-cells $U=\bigcup_{i=1}^NK^{(i)}$, one defines its boundary $\partial U$ to be the set of all vertex points that intersect its complement.

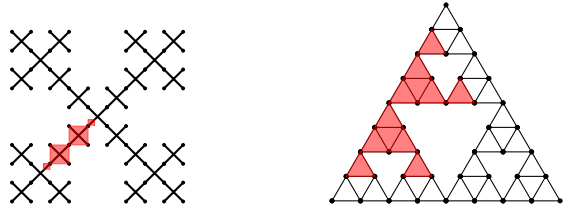
\begin{figure}[H]
\centering
\begin{tabular}{ccc}
\begin{tikzpicture}[scale=1/4]
\foreach \c/\d in {0/0, 0/6, 3/3, 6/0, 6/6}{
\foreach \a/\b in {0/0, 0/2, 1/1, 2/0, 2/2}{
\fill[shift={(\c,\d)}] (\a,\b) circle (3pt);
\fill[shift={(\c,\d)}] (\a,\b+1) circle (3pt);
\fill[shift={(\c,\d)}] (\a+1,\b+1) circle (3pt);
\fill[shift={(\c,\d)}] (\a+1,\b) circle (3pt);
\draw[thick,shift={(\c,\d)}] (\a,\b) -- (\a+1,\b+1) (\a,\b+1) -- (\a+1,\b);
}
}
\filldraw[red,semitransparent] (2,2) --++ ($(90:1)$) --++ ($(0:1)$) --++ ($(90:1)$) --++ ($(0:1)$) --++ ($(270:1)$) --++ ($(180:1)$) --++ ($(270:1)$) --++ ($(180:1)$);
\filldraw[red,semitransparent] (5/3,5/3) --++ ($(90:1/3)$) --++ ($(0:1/3)$) --++ ($(270:1/3)$);
\filldraw[red,semitransparent] (4,4) --++ ($(90:1/3)$) --++ ($(0:1/3)$) --++ ($(270:1/3)$);
\end{tikzpicture}
&\hspace*{3em}&
\begin{tikzpicture}[scale=1.5]
\tikzstyle{every node}=[draw,circle,fill=black,minimum size=1.5pt, inner sep=0pt]
\draw ($(0:0)$) node () {} --++ ($(0:2/8)$) node () {} --++ ($(120:2/8)$) node () {} --++ ($(240:2/8)$) node () {};
\draw ($(0:2/8)$) node () {} --++ ($(0:2/8)$) node () {} --++ ($(120:2/8)$) node () {} --++ ($(240:2/8)$) node () {};
\draw ($(60:2/8)$) node () {} --++ ($(0:2/8)$) node () {} --++ ($(120:2/8)$) node () {} --++ ($(240:2/8)$) node () {};
\foreach \a in {0,60}{
\draw ($(\a:2/4)$) node () {} --++ ($(0:2/8)$) node () {} --++ ($(120:2/8)$) node () {} --++ ($(240:2/8)$) node () {};
\foreach \b in{0,60}{
\draw ($(\a:2/4)+(\b:2/8)$)node () {} --++ ($(0:2/8)$) node () {} --++ ($(120:2/8)$) node () {} --++ ($(240:2/8)$) node () {};
}
}
\foreach \c in{0,60}{
\draw ($(\c:2/2)$) node () {} --++ ($(0:2/8)$) node () {} --++ ($(120:2/8)$) node () {} --++ ($(240:2/8)$) node () {};
\draw ($(\c:2/2)+(0:2/8)$) node () {} --++ ($(0:2/8)$) node () {} --++ ($(120:2/8)$) node () {} --++ ($(240:2/8)$) node () {};
\draw ($(\c:2/2)+(60:2/8)$) node () {} --++ ($(0:2/8)$) node () {} --++ ($(120:2/8)$) node () {} --++ ($(240:2/8)$) node () {};
\foreach \a in {0,60}{
\draw ($(\c:2/2)+(\a:2/4)$) node () {} --++ ($(0:2/8)$) node () {} --++ ($(120:2/8)$) node () {} --++ ($(240:2/8)$) node () {};
\foreach \b in{0,60}{
\draw ($(\c:2/2)+(\a:2/4)+(\b:2/8)$)node () {} --++ ($(0:2/8)$) node () {} --++ ($(120:2/8)$) node () {} --++ ($(240:2/8)$) node () {};
}
}
}
\filldraw[red,semitransparent] ($(0:0)+(60:1/4)$) --++ ($(60:6/4)$) --++ ($(300:1/4)$) --++ ($(180:1/4)$) --++ ($(300:1/2)$) --++ ($(60:1/4)$) --++ ($(300:1/4)$) --++ ($(180:3/4)$) --++ ($(300:3/4)$) --++ ($(180:1/4)$) --++ ($(60:1/4)$) --++ ($(180:1/2)$) --++ ($(300:1/4)$) --++ ($(180:1/4)$);
\end{tikzpicture}
\end{tabular}
\caption{Unions of cells in the Vicsek set (left) and in the Sierpinski gasket (right)}
\label{F:U_cells_example_intro}
\end{figure}

\begin{thm}\label{T:liminf_Ucells_intro}
Let $(X,d,\mu)$ be an unbounded nested fractals as described in Section~\ref{S:setup}, with length scaling factor $L$, Hausdorff dimension $d_h$ and walk dimension $d_w$. There exist positive and bounded periodic functions $\Phi$ and $\Psi$ with period $L^{-d_w}$, respectively $L^{-1}$, such that for any finite union of cells $U\subset X$,
\begin{equation}\label{E:liminf_Ucells_intro}
\lim_{t\to 0^+}\Phi(-\ln t)\frac{1}{t^{d_h/d_w}}\mathcal{M}_{\mathbf{1}_U}(t)=\lim_{r\to 0^+}\Psi(\ln r)\frac{1}{r^{d_h}}\widetilde{\mathcal{M}}_{\mathbf{1}_U}(r)=|\partial U|.
\end{equation}
Here, $|\partial U|$ equals the number of points in the boundary of $U$. In the case of the unbounded Vicsek set or Sierpinski gasket, the function $\Psi$ is non-constant.
\end{thm}

While this oscillatory behavior was expected in view of the on-diagonal oscillations of the heat kernel at small scales in these settings, see e.g.~\cite{Ham11,Kaj13}, the non-uniqueness of the BV measure is less straightforward to obtain. In fact, we can presently prove the function $\Psi$ in~\eqref{E:liminf_Ucells_intro} to be non-constant, see Section~\ref{SS:KS_non-limit}, thus settling a question raised in~\cite[Remark 4.23]{BV3}. 
However, the methods currently available fail to show that property for the function $\Phi$, which is also still open in the case of other heat-kernel related functionals~\cite{Ham11}. Besides, the nature of the functional $\mathcal{M}_f(t)$ seems to make the techniques that successfully led Kajino to prove on-diagonal oscillations in~\cite{Kaj13} not applicable in this case. The question remains the subject of future investigations.

\begin{conj}\label{conj1}
The periodic function $\Phi$ in~\eqref{E:liminf_Ucells_intro} is non-constant.
\end{conj}

{Proving Conjecture \ref{conj1} would illustrate further the stark contrast between the $L^1$ theory and the $L^2$ theory on fractals since for any function $f$ in the domain of the Dirichlet form $\mathcal{E}$ associated with  the heat kernel $p_t(x,y)$ one has the following exact limit
\[
\lim_{t \to 0^+} \frac{1}{2t} \int_{X} \int_{X} |f (x) -f(y)|^2 p_t(x,y) \, d\mu (x) d\mu(y)=\mathcal{E}(f,f).
\]}
  The paper is organized as follows: Section 2 describes the unbounded nested fractals considered as underlying spaces. Section~\ref{S:HS_functional} deals with the part of the proof of Theorem~\ref{T:liminf_Ucells_intro} concerning the heat kernel functional $\mathcal{M}_{\mathbf{1}_U}$, while Section 4 provides the proof corresponding to the Korevaar-Schoen functional $\widetilde{\mathcal{M}}_{\mathbf{1}_U}$. This last section also contains specific examples and proves the fact that $\Psi$ is non-constant, which in particular implies the non-uniqueness of the BV measures.
 
\section{Notations and set up}\label{S:setup}

\subsection{Compact nested fractals}

To set the framework and notation throughout the paper, this section briefly recalls the construction of planar simple nested fractals as introduced by Lindstr\o m in~\cite{Lin90}, see also~\cite{Kum93}. For $L>1$, an $L$-similitude is a map $\psi:\mathbb R^d\to \mathbb R^d$ such that 
\[
\psi(x)=L^{-1}A(x)+b,
\] 
where $A$ is a unitary linear map and $b\in\mathbb R^d$. The factor $L^{-1}$ is called the contraction ratio of $\psi$. 
Given a collection of $L$-similitudes $\{\psi_i\}_{i=1}^M$ in $\mathbb R^d$, there exists a unique nonempty compact set $K\subset\mathbb R^d$ such that 
\[
K=\bigcup_{i=1}^M \psi_i(K)=:\Psi(K).
\]
Each map $\psi_i$ has a unique fixed point $q_i$ and we denote by $V:=\{q_i\}_{i=1}^M$ the set of all fixed points. A fixed point $x\in V$ is called an essential fixed point if there exist $y\in V$ and $i\ne j$ such that $\psi_i(x)=\psi_j(y)$; we denote by $V^{(0)}$ the set of all essential fixed points. 
For any $n\in \mathbb N$, we further write $V^{(n)}:=\Psi^{n}(V^{(0)})$ and 
\[
V^{(\infty)}:=\bigcup_{n\in\mathbb{N}}V^{(n)}.
\]
Finally, we define the word spaces $W_n:=\{1,2,\cdots, M\}^n$,  for each $n\geq 1$, and $W_{\infty}:=\{1,2,\cdots, M\}^{\mathbb N}$. Each finite word $w=(i_1, \cdots, i_n)\in W_n$ addresses the map $\psi_w:=\psi_{i_1}\circ \cdots \circ \psi_{i_n}$ and the set $A_w:=\psi_w(A)$ for any $A\subseteq K$. Deviating from the original terminology in~\cite{Lin90} to the currently more established one, $K_w$ will be called an $n$-cell with set of vertices $V_w^{(0)}$. 

\begin{defn}\label{def:nested}
Let $(K, \psi_1, \cdots, \psi_M)$ be as described above. The set $K$ is called a nested fractal if the following conditions are satisfied:
\begin{enumerate}
\item $|V^{(0)}|\ge 2$;

\item (Connectivity) For any $i,j\in W_1$, there exists a sequence of 1-cells $V_{i_0}^{(0)},\cdots, V_{i_k}^{(0)}$ such that $i_0=i$, $i_k=j$ and $V_{i_{r-1}}^{(0)} \cap V_{i_r}^{(0)} \ne 0$, for $1\le r\le k$;

\item (Symmetry) For any $x,y\in V^{(0)}$, the reflection in the hyperplane $H_{xy}=\{z\in\mathbb{R}^d:|x-z|=|y-z|\}$ maps $n$-cells to $n$-cells;

\item (Nesting) For any $w,v\in W_n$ and $w\ne v$, $K_w\cap K_v= V_w^{(0)} \cap V_v^{(0)}$;

\item (Open set condition) There exists a non-empty bounded open set $U$ such that $\psi_i(U)$, $1\le i\le M$, are disjoint and $\Psi(U)\subset U$.
\end{enumerate}
\end{defn}

In addition, we will require that any two $n$-cells intersect at most at one point, i.e.
\begin{equation*}
|V^{(0)}_v\cap V^{(0)}_w|\in\{0,1\}
\end{equation*}
for any $v,w\in W_{n}$ and $n\in \mathbb{N}$. This property may possibly follow from the conditions imposed for nested fractals and it is still an open question whether that is actually the case, see~\cite[Remark 5.25]{Bar95}.

\medskip

The parameters $L$ and $M$ are called the \emph{length scaling factor}, respectively the \emph{mass scaling factor} of $K$. In particular, with respect to the Euclidean distance $d(x,y)$, 
\[
d(\psi_w(x),\psi_w(y))=L^{-n}d(x,y),
\]
for any $x,y\in K$ and $w\in W_n$, whereas 
the normalized Hausdorff measure $\mu$ on $K$ satisfies
\[
\mu(K_w)=\mu(\psi_w(K))=M^{-n}
\]
for all $w\in W_n$. The \emph{Hausdorff dimension} of $K$ is thus given by
\[
d_h=\frac{\log M}{\log L}.
\]
The two main examples in the present paper are the Vicsek set and the Sierpinski gasket illustrated in Figure~\ref{F:VS_SG}.

\begin{figure}[H]
\centering
\begin{tabular}{ccc}
\begin{tikzpicture}[scale=1/4]
\foreach \c/\d in {0/0, 0/6, 3/3, 6/0, 6/6}{
\foreach \a/\b in {0/0, 0/2, 1/1, 2/0, 2/2}{
\fill[shift={(\c,\d)}] (\a,\b) circle (3pt);
\fill[shift={(\c,\d)}] (\a,\b+1) circle (3pt);
\fill[shift={(\c,\d)}] (\a+1,\b+1) circle (3pt);
\fill[shift={(\c,\d)}] (\a+1,\b) circle (3pt);
\draw[thick,shift={(\c,\d)}] (\a,\b) -- (\a+1,\b+1) (\a,\b+1) -- (\a+1,\b);
}
}
\end{tikzpicture}
&\hspace*{3em}&
\begin{tikzpicture}[scale=1/20]
\foreach \c in {0,1,2} {
\foreach \b in {0,1,2} {
\foreach \a in {0,1,2} {
\foreach \e in {0,1,2} {
\foreach \f in {0,1,2} {
\draw[] ($(90+120*\f:16)+(90+120*\e:8) + (90+120*\c:4)+(90+120*\b:2)+(90+120*\a:1)+(90:1)$) -- ($(90+120*\f:16)+(90+120*\e:8) + (90+120*\c:4)+(90+120*\b:2)+(90+120*\a:1)+(210:1)$) -- ($(90+120*\f:16)+(90+120*\e:8) + (90+120*\c:4)+(90+120*\b:2)+(90+120*\a:1)+(330: 1)$)--cycle; 
}
}
}
}
}
\end{tikzpicture}
\end{tabular}
\caption{The Vicsek set (left) and the Sierpinski gasket (right)}
\label{F:VS_SG}
\end{figure}
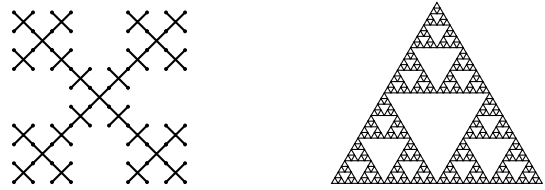

\begin{Example}[Sierpinski gasket]
\

\noindent 
Let $q_1=0, q_2=1, q_3=e^{\frac{i\pi}3}$ be three points in $\mathbb R^2=\mathbb C$ regarded as the vertices of an equilateral triangle of side length one. Further, define $\psi_i(z):=\frac12(z-q_i)+q_i$ for $i=1,2,3$. The Sierpi\'nski gasket $K_{\rm SG}$ is the unique non-empty compact set such that 
\[
K_{\rm SG}=\bigcup_{i=1}^3 \psi_i(K_{\rm SG}).
\]
Its associated standard measure $\mu$ is a normalized Hausdorff measure that satisfies
\[
\mu(\psi_{i_1} \circ \cdots \circ \psi_{i_n} (K_{SG}))=3^{-n}
\]
for $i_1, \cdots, i_n \in \{1,2,3\}$. 
Thus, $K_{SG}$ is a nested fractal with scaling factor $L_{\rm SG}=2$, and mass scaling factor $M_{\rm SG}=3$. In particular, $d_h=\frac{\log 3}{\log 2}$ and $d_w=\frac{\log 5}{\log 2}$.
\end{Example}

\begin{Example}[Vicsek set] \label{VS}
\

\noindent Let $q_1=(0,0)$, $q_2=(0,1)$, $q_3=(1,0)$ and $q_4=(1,1)$ be the corners of a unit square and let $q_5=(1/2,1/2)$. For each $1\le i\le 5$ define the map  $\psi_i(z):=\frac13(z-q_i)+q_i$. The Vicsek set $K_{\rm VS}$ is the unique non-empty compact set such that 
\[
K_{\rm VS}=\bigcup_{i=1}^5 \psi_i(K_{\rm VS}).
\]
Its associated standard measure $\mu$ is a normalized Hausdorff measure that satisfies
\[
\mu(\psi_{i_1} \circ \cdots \circ \psi_{i_n} (K_{\rm VS}))=5^{-n}
\]
for each $i_1, \cdots, i_n \in \{1,2,3, 4,5\}$. Thus, the Vicsek set $K_{\rm VS}$ is a nested fractal with scaling factor $L_{\rm VS}=3$ and mass scaling factor $M_{\rm VS}=5$. In particular, $d_h=\frac{\log 5}{\log 3}$ and $d_w=\frac{\log 15}{\log 3}$.
\end{Example}

\subsection{Unbounded nested fractals}

An unbounded nested fractal arises as a blow-up of a compact nested fractal. We refer to~\cite{Kum93, Bar13, KOP19} and also~\cite{Str98} for different constructions.
Without loss of generality, we will assume from now on that $\psi_1=L^{-1} x$ and consider the unbounded nested fractal $K^{\langle\infty\rangle}$ defined as
\[
K^{\langle\infty\rangle}:=\bigcup_{n=1}^{\infty} K^{\langle n\rangle},
\]
where $K^{\langle n\rangle}=L^n K$. Its associated set of essential fixed points is defined analogously as $V_0^{\langle \infty\rangle}=\bigcup_{n=0}^{\infty}V_n$, where $V_n=L^n V^{(0)}$, and $V_n^{\langle\infty\rangle}=L^{n}V^{(n)}$ 
The associated standard Hausdorff measure, denoted by $\mu^{\langle\infty\rangle}$, satisfies $\mu^{\langle\infty\rangle}(K^{\langle n\rangle})=M^n$ and it is $d_h$-Ahlfors regular, i.e. 
\begin{equation}\label{eq:volume}
cr^{d_h} \le \mu^{\langle\infty\rangle}(B(x,r)) \le  Cr^{d_h}
\end{equation}
for any $x \in K^{\langle\infty\rangle}$ and $r \ge 0$.

\subsection{Brownian motion and heat kernels}
Brownian motion $(X_t)_{t \ge 0}$ on a simple planar nested fractal $K$ and its associated unbounded fractal $K^{\langle\infty\rangle}$ was rigororusly constructed 
in~\cite{Kus89,Lin90}, see also~\cite{Fuk92,Kum93}.
This Brownian motion is a Hunt diffusion process whose
associated heat semigroup $\{P_t^{K^{\langle\infty\rangle}}\}_{t\ge 0}$ admits a jointly continuous heat kernel with respect to the $d_h$-dimensional Hausdorff measure $\mu^{\langle\infty\rangle}$. We denote this kernel by $p_t^{K^{\langle\infty\rangle}}(x,y)$. 
It satisfies the scaling property
\begin{equation}
    p_t^{K^{\langle\infty\rangle}}(x,y)=Mp^{K^{\langle\infty\rangle}}_{L^{d_w}t} (Lx,Ly)
\end{equation}
for any $x,y \in K^{\langle\infty\rangle}$ and $t>0$, as well as the 
sub-Gaussian estimates
\begin{equation}\label{E:HK_subG_Kinfty}
c_1t^{-d_{h}/d_{w}}\exp\biggl(-c_2\Bigl(\frac{d(x,y)^{d_{w}}}{t}\Bigr)^{\frac{1}{d_{J}-1}}\biggr) \le p^{K^{\langle\infty\rangle}}_t(x,y)
\le  c_3 t^{-d_{h}/d_{w}}\exp\biggl(-c_4\Bigl(\frac{d(x,y)^{d_{w}}}{t}\Bigr)^{\frac{1}{d_{J}-1}}\biggr)
\end{equation}
for every $(x,y)\in K^{\langle\infty\rangle}\times K^{\langle\infty\rangle}$ and $t >0$, c.f.~\cite[Theorem 5.2, Theorem 5.5]{Kum93}.
The parameter $d_J>0$ is the dimension of the Brownian motion $X_t$ in the shortest path metric, while the parameter $d_w$ is the dimension of $X_t$ in the Euclidean metric. The latter is usually called the walk dimension of $K^{\langle\infty\rangle}$.

%

\medskip

The estimates~\eqref{E:HK_subG_Kinfty} where obtained under a particular assumption~\cite[Assumption 2.2]{Kum93} on the underlying metric that is satisfied in the case of planar simple nested fractals, c.f.~\cite[Appendix]{KOP19}. Although in general $d_J\neq d_w$, we will assume throughout of the paper that $d_J=d_w$ since that is the case in the examples that are presented. The main results are likely true after possibly minor changes without the assumption $d_J=d_W$, however the paper focuses on building most immediate counterexamples, leaving a generalization for possible future work.

\medskip

The estimates~\eqref{E:HK_subG_Kinfty} also provide the following bound on the exit time of balls proved in~\cite[Lemma 2.3]{Ham11} which will be used in Lemma~\ref{lem2}. Here and in the sequel we define $\tau_{U^c}:=\inf\{t>0\colon X_t\notin U\}$ for any $U\subset K^{\langle\infty\rangle}$.

\begin{lem}\label{L:hitting_time}
For any closed set $U\subset K^{\langle\infty\rangle}$ there exist $C,c>0$ such that
\begin{equation}\label{E:hitting time}
\mathbb{P}_x(\tau_{U^c}<t)\leq Ce^{-c\big(\frac{d(x,U^c)^{d_w}}{t}\big)^{\frac{1}{d_w-1}}}
\end{equation}
for any $x\in U$ and $t>0$.
\end{lem}

In addition, Section~\ref{S:HS_functional} will deal with the process killed outside of a compact $F\subset  K^{\langle\infty\rangle}$, whose associated Dirichlet heat kernel we denote by $p_t^F(x,y)$. For any $A \subset F$ and $x \in F$ this kernel satisfies
\begin{equation}\label{E:def_Dirichlet_HK}
\mathbb{P}_x( X_t \in A, t \le \tau_{F^c} )=\int_A p_t^F(x,y) d\mu^{\langle\infty\rangle}(y)
\end{equation}
and the scaling invariance property 
\begin{align}\label{scaling}
p_t^{\psi_w(F)}(\psi_w(x),\psi_w(y))=M^n p^F_{L^{nd_w}t} (x,y)
\end{align}
for any $x,y \in F$, and $w\in W_n$, $n\geq 1$, c.f.~\cite[Theorem 3.2, (ii)]{Ham11}.

\subsection{Korevaar-Schoen-Sobolev and BV spaces on fractals}\label{S:KS_BV}

The BV measures investigated in the present paper were introduced in~\cite{BV3} to study the case $p=1$ of the heat semigroup based Besov spaces $\mathbf{B}^{p, \alpha}(K^{\langle\infty\rangle})$ defined below. These spaces also admit a Korevaar-Schoen-Sobolev characterization, c.f. Theorem~\ref{main:BV3}, which will be key to prove in Section~\ref{SS:KS_non-limit} the non-trivial oscillations of the BV measures on the Vicsek set and the Sierpinski gasket.

\begin{defn}[Heat semigroup Besov classes]
For any $p \ge 1$ and $\alpha>0$, define the heat semigroup Besov seminorm
\begin{equation*}
    \|f\|_{p,\alpha}:=\sup_{t>0} \, t^{-\alpha} \left(\int_{K^{\langle\infty\rangle}}\int_{K^{\langle\infty\rangle}} p_t^{K^{\langle\infty\rangle}}(x,y)|f(x)-f(y)|^pd\mu^{\langle\infty\rangle}(x)d\mu^{\langle\infty\rangle}(y)\right)^{1/p}
\end{equation*}
and the heat semigroup based Besov class
\[
\mathbf{B}^{p, \alpha}(K^{\langle\infty\rangle}):=\left\{f\in L^p(K^{\langle\infty\rangle},\mu^{\langle\infty\rangle}), \|f\|_{p,\alpha}<\infty\right\}.
\]
\end{defn}
It was proved in~\cite[Proposition 4.14]{BV1} that $(\mathbf{B}^{p, \alpha}(K^{\langle\infty\rangle}), \| \cdot \|_{p,\alpha}+\|\cdot\|_{L^p(K,\mu)})$ is a complete Banach space. 

\begin{defn}[BV class]\label{def:pKS}
For any $r>0$ and $f\in L^1(K^{\langle\infty\rangle},\mu^{\langle\infty\rangle})$ let
\begin{equation*}
\widetilde{\mathcal{M}}_f(r):=\frac{1}{r^{d_h}}\int_{K^{\langle\infty\rangle}}\int_{B(x,r)\cap K^{\langle\infty\rangle}}|f(x)-f(y)|\,d\mu^{\langle\infty\rangle}(y)\,d\mu^{\langle\infty\rangle}(x).
\end{equation*}
The BV class is defined as 
\[
BV(K^{\langle\infty\rangle}):=
\Big\{f\in L^1(K^{\langle\infty\rangle},\mu^{\langle\infty\rangle})\colon \limsup_{r\to 0^+} \frac{1}{r^{d_h}}\widetilde{\mathcal{M}}_f(r)<\infty\Big\}. 
\]
\end{defn}
In the framework of {unbounded} nested fractals, the space $BV(K^{\langle\infty\rangle})$ is known to contain linear combinations of indicator functions of $n$-cells, see~\cite[Theorem 5.1]{BV3} and is therefore dense in $L^1(K^{\langle\infty\rangle},\mu^{\langle\infty\rangle})$. 
The main motivation for the present work is the close relation between the BV class $BV(K^{\langle\infty\rangle})$ and the Besov classes defined above. Introducing the functional
\begin{equation}\label{E:def_Mf}
\mathcal{M}_f (t) :=\int_{K^{\langle\infty\rangle}} \int_{K^{\langle\infty\rangle}} |f (x) -f(y)| p^{K^{\langle\infty\rangle}}_t(x,y) \, d\mu^{\langle\infty\rangle} (x) d\mu^{\langle\infty\rangle}(y),
\end{equation}
for $f \in L^1(K^{\langle\infty\rangle},\mu^{\langle\infty\rangle})$, an immediate consequence of~\cite[Theorem 4.24]{BV3} is the following characterization of BV.

\begin{thm}[Heat semigroup characterization of BV functions]\label{main:BV3}
For an unbounded nested fractal $K^{\langle\infty\rangle}$ with Hausdorff dimension $d_h$ and walk dimension $d_w$,
\[
BV(K^{\langle\infty\rangle})=\mathbf{B}^{1, d_h/d_w}(K^{\langle\infty\rangle})
\]
with equivalent norms. More precisely, there exist constants $c,C>0$ such that for every $f \in BV(K^{\langle\infty\rangle})$,

\begin{equation*}
c \, \sup_{t>0} \, t^{-d_h/d_w} 
\mathcal{M}_f (t)
\le  \sup_{r >0} \frac{1}{r^{d_h}}\widetilde{\mathcal{M}}_f(r)
\le  C \,  \liminf_{t \to 0^+} \, t^{-d_h/d_w} 
\mathcal{M}_f (t).
\end{equation*}
\end{thm}

\subsection{A renewal lemma}
One of the main ingredients in the proof of the main results in Theorem~\ref{T:liminf_Ucells_heat} and Theorem~\ref{T:liminf_Ucells_KS} is the following renewal lemma adapted from~\cite[Lemma 3.5]{Ham11}. Here and throughout the paper, given any two functions $f, g\colon (0,+\infty) \to \mathbb{R}$, we will write
\[
f(t) \simeq g(t)
\]
if there exist constants $c,C>0$ such that for every $t \in (0,1]$,
\begin{equation}\label{E:def_simeq}
| f(t)-g(t) | \le c e^{ -Ct^{-\frac{1}{d_w-1}}}.
\end{equation}

\begin{lem}[Renewal lemma~{\rm\cite{Ham11}}]\label{L:renewal_lemma}
Let $\alpha, \beta >0$ with $\alpha \ge 1$, $\beta <1$ and let $f\colon (0,+\infty) \to (0,+\infty)$ be a continuous bounded function such that
\[
f(t) \simeq \alpha f(\beta t).
\]
Then, there  exists a periodic function $\theta$ with period $\beta$ such that
\[
f(t)=t^{-\frac{\ln \alpha }{\ln \beta }} \theta(- \ln t) +o\left( t^{-\frac{\ln \alpha }{\ln \beta }}\right)
\]
exists as $t\to 0^+$.
\end{lem}

\begin{proof}
Let $g(t)=e^{\gamma t} f(e^{-t})$ where $\gamma =-\frac{\ln \alpha }{\ln \beta }$ so that $f(t)= t^\gamma g(- \ln t)$. For $t \in (0,1]$, by definition of the relation $\simeq$ in~\eqref{E:def_simeq}, we have
\[
\left| t^\gamma g(- \ln t) - \alpha \beta^ \gamma t^\gamma  g(- \ln t -\ln \beta) \right| \le c e^{ -Ct^{-\frac{1}{d_w-1}}}
\]
and therefore
\[
\left| t^\gamma g(- \ln t) - \alpha \beta^ \gamma t^\gamma  g(- \ln t -\ln \beta) \right| \le c e^{ -Ct^{-\frac{1}{d_w-1}}}
\]
which implies
\[
\left| g(- \ln t) -  g(- \ln t -\ln \beta) \right| \le c t^{-\gamma} e^{ -Ct^{-\frac{1}{d_w-1}}}.
\]
Thus, for $t \ge 1$,
\[
\left| g(- \ln t) -  g(- \ln t -\ln \beta) \right|  \le C t^{-\gamma}
\]
and hence 
\[
\left| g( t) -  g(t -\ln \beta) \right| \le c e^{-C | t |}
\]
for some constants $c,C>0$ and $t \in \mathbb R$. 
Setting
\[
\theta (t):=\sum_{k=-\infty}^{+\infty} ( g(t -k \ln \beta) -  g(t -(k+1) \ln \beta)),
\]
the lemma follows.
\end{proof}

\section{Heat semigroup functional}\label{S:HS_functional}

The aim of this section is to prove the first part of the main result, Theorem~\ref{T:liminf_Ucells_intro}, which involves the heat semigroup functional $\mathcal{M}_f(t)$ from~\eqref{E:def_Mf}.

\begin{thm}\label{T:liminf_Ucells_heat}
Let $K^{\langle\infty\rangle}$ be an unbounded nested fractal as in Section~\ref{S:setup} with Hausdorff dimension $d_h$ and walk dimension $d_W$. There exists a bounded periodic function $\Phi: (0,+\infty)\to [a,b]$ with period $L^{-d_w}$ and $0 < a  \le b$ such that, for any finite union of $n$-cells $U\subset K^{\langle\infty\rangle}$,
\begin{equation}\label{E:liminf_Ucells_heat}
\lim_{t\to 0^+}\Phi(-\ln t)\frac{1}{t^{d_h/d_w}}\mathcal{M}_{\mathbf{1}_U}(t)=|\partial U|,
\end{equation}
where $|\partial U|$ denotes the number of points in the boundary of $U$.
\end{thm}

%
The proof of Theorem~\ref{T:liminf_Ucells_heat} follows a strategy developed by Hambly in~\cite{Ham11} and is divided into several lemmas presented in the next section.

\subsection{Preliminary lemmas}
To show Theorem~\ref{T:liminf_Ucells_heat} we introduce the auxiliary functional
\begin{equation}\label{E:def_heat_sets}
\mathcal{M}_{A,B}(t):=\int_{A}\int_{B}p^{K^{\langle\infty\rangle}}_t(x,y)\,d\mu^{\langle\infty\rangle}(y)\,d\mu^{\langle\infty\rangle}(x)
\end{equation}
for any two compact $A,B \subset K^{\langle\infty\rangle}$ and $t>0$. As a consequence of Lemma~\ref{L:cell_pair_asymp_KS}, the asymptotic behavior of the functional $\mathcal{M}_{A,B}(t)$ will mainly depend on the behavior of the heat kernel near $A\cup B$. In the sequel, the $r$-neighborhood of any $A \subset K^{\langle\infty\rangle}$ will be denoted by
\[
A_r:=\{ x \in K^{\langle\infty\rangle}\colon d (x,A) \le r\}.
\]

\begin{lem}\label{lem2}
For any compact sets $A,B \subset K^{\langle\infty\rangle}$ and $r >0$,
\[
\mathcal{M}_{A,B}(t) \simeq \int_{A}\int_{B}p^{(A \cup B)_r}_t(x,y)\,d\mu^{\langle\infty\rangle}(y)\,d\mu^{\langle\infty\rangle}(x).
\]
\end{lem}

\begin{proof}
By definition of Dirichlet kernel, c.f.~\eqref{E:def_Dirichlet_HK},
\[
\int_{A}\int_{B}p^{(A \cup B)_r}_t(x,y)\,d\mu^{\langle\infty\rangle}(y)\,d\mu^{\langle\infty\rangle}(x)=\int_A \mathbb{P}_x( X_t \in B, t \le \tau_{((A \cup B)_r)^c} )\, d\mu^{\langle\infty\rangle}(x).
\]
Moreover,
\begin{align*}
\mathbb{P}_x( X_t \in B)=\mathbb{P}_x( X_t \in B, t \le \tau_{((A \cup B)_r)^c} )+\mathbb{P}_x( X_t \in B, t > \tau_{((A \cup B)_r)^c} ),
\end{align*}
hence
\begin{equation*}
 \left| \mathbb{P}_x( X_t \in B)-\mathbb{P}_x( X_t \in B, t \le \tau_{((A \cup B)_r)^c} )\right| \le \mathbb{P}_x(  t > \tau_{((A \cup B)_r)^c} )
 \leq Ce^{-c\big(\frac{d(x,((A \cup B)_r)^c)^{d_w}}{t}\big)^{\frac{1}{d_w-1}}},
\end{equation*}
where the last inequality follows from Lemma~\ref{L:hitting_time}. 
Since $ d(x,((A \cup B)_r)^c) \ge r$ for any $ x \in A$, we obtain
\[
\mathbb{P}_x(  t > \tau_{((A \cup B)_r)^c} ) \leq Ce^{-c\big(\frac{r^{d_w}}{t}\big)^{\frac{1}{d_w-1}}}
\]
which yields
\begin{multline*}
 \left| \int_{A}\int_{B}p^{K^{\langle\infty\rangle}}_t(x,y)\,d\mu^{\langle\infty\rangle}(y)\,d\mu^{\langle\infty\rangle}(x) -
\int_{A}\int_{B}p^{(A \cup B)_r}_t(x,y)\,d\mu^{\langle\infty\rangle}(y)\,d\mu^{\langle\infty\rangle}(x) \right| \\
\leq C\mu^{\langle\infty\rangle}(A) e^{-c\big(\frac{r^{d_w}}{t}\big)^{\frac{1}{d_w-1}}}.
\end{multline*}
\end{proof}

The next step consists in proving a scaling and a localization property of the functional $\mathcal{M}_{A,B}(t)$ that will allow us to compare its behavior across different levels. In particular, when the sets $A,B$ are well separated, the associated functional $\mathcal{M}_{A,B}(t)$ becomes asymptotically negligible.

\begin{lem}[Scaling lemma]\label{L:scaling lemma}
For any $w\in W_n$, $n\geq 1$, and compact sets $A,B \subset K^{\langle\infty\rangle}$,
\begin{equation}\label{E:scaling_M2cells}
\mathcal{M}_{A,B}(t) \simeq M^n\mathcal{M}_{\psi_w(A),\psi_w(B)}(L^{-nd_w}t).
\end{equation}
\end{lem}

\begin{proof}
From Lemma \ref{lem2} and the scaling property of the heat kernel~\eqref{scaling},
\begin{align*}
\mathcal{M}_{\psi_w(A),\psi_w(B)}(t) & \simeq  \int_{\psi_w(A)}\int_{\psi_w(B)}p^{(\psi_w(A) \cup \psi_w(B))_r}_t(x,y)\,d\mu^{\langle\infty\rangle}(y)\,d\mu^{\langle\infty\rangle}(x) \\
 &  \simeq M^{-2n} \int_{A}\int_{B}p^{(\psi_w(A) \cup \psi_w(B))_r}_t(\psi_w(x),\psi_w(y))\,d\mu^{\langle\infty\rangle}(y)\,d\mu^{\langle\infty\rangle}(x) \\
 &  \simeq M^{-2n} \int_{A}\int_{B}p^{\psi_w ((A \cup B)_{L^{-n}r})}_t(\psi_w(x),\psi_w(y))\,d\mu^{\langle\infty\rangle}(y)\,d\mu^{\langle\infty\rangle}(x) \\
 &  \simeq M^{-n} \int_{A}\int_{B}p^{(A \cup B)_{L^{-n}r}}_{L^{nd_w}t}(x,y)\,d\mu^{\langle\infty\rangle}(y)\,d\mu^{\langle\infty\rangle}(x) \\
 & \simeq M^{-n} \int_{A}\int_{B}p^{K^{\langle\infty\rangle}}_{L^{nd_w}t}(x,y)\,d\mu^{\langle\infty\rangle}(y)\,d\mu^{\langle\infty\rangle}(x) \\
 & \simeq M^{-n} \mathcal{M}_{A,B}(L^{nd_w}t).
\end{align*}
\end{proof}

\begin{lem}[Localization lemma]\label{L:localization}
Let $A,B\subset K^{\langle\infty\rangle}$ be compact and let $\tilde A \subset A$, $\tilde B \subset B$ be such that $d(\tilde A , B \setminus  \tilde B) >0$, $d(A \setminus \tilde A , B \setminus  \tilde B) >0$ and $d(A \setminus \tilde A ,  \tilde B) >0$. Then,
\begin{equation*}
\mathcal{M}_{A,B}(t) \simeq \mathcal{M}_{\tilde A,\tilde B}(t).
\end{equation*}
In particular, if $d( A , B ) >0$,
\[
\mathcal{M}_{A,B}(t) \simeq 0.
\]
\end{lem}

\begin{proof}
Decomposing $A \times B$ accordingly,
\begin{align*}
\mathcal{M}_{A,B}(t) &= \int_{A}\int_{B}p^{K^{\langle\infty\rangle}}_t(x,y)\,d\mu^{\langle\infty\rangle}(y)\,d\mu^{\langle\infty\rangle}(x) \\
   &= \int_{\tilde A}\int_{ \tilde B}p^{K^{\langle\infty\rangle}}_t(x,y)\,d\mu^{\langle\infty\rangle}(y)\,d\mu^{\langle\infty\rangle}(x) + \int_{\tilde A}\int_{ B \setminus  \tilde B}p^{K^{\langle\infty\rangle}}_t(x,y)\,d\mu^{\langle\infty\rangle}(y)\,d\mu^{\langle\infty\rangle}(x) \\
   &+  \int_{A \setminus \tilde A}\int_{ B \setminus  \tilde B}p^{K^{\langle\infty\rangle}}_t(x,y)\,d\mu^{\langle\infty\rangle}(y)\,d\mu^{\langle\infty\rangle}(x) +  \int_{A \setminus \tilde A}\int_{  \tilde B}p^{K^{\langle\infty\rangle}}_t(x,y)\,d\mu^{\langle\infty\rangle}(y)\,d\mu^{\langle\infty\rangle}(x).
\end{align*}
The result now follows by observing that the upper heat kernel estimate~\eqref{E:HK_subG_Kinfty} implies
\[
p^{K^{\langle\infty\rangle}}_t(x,y) \le  c t^{-d_{h}/d_{w}}\exp\biggl(-C \Bigl(\frac{d(U_1,U_2)^{d_{w}}}{t}\Bigr)^{\frac{1}{d_{w}-1}}\biggr)
\]
for any sets $U_1,U_2\subset K^{\langle\infty\rangle}$ and $x \in U_1$ and $y\in U_2$.
\end{proof}



\begin{remark}
The results presented in this section actually hold in a much larger class of fractals than that of unbounded nested, including infinitely ramified. Proceeding further will however require to restrict ourselves to the framework of unbounded nested fractals.
\end{remark}

\subsection{Proof of Theorem~\ref{T:liminf_Ucells_heat}}
We will now combine the results in the previous sections to prove Theorem~\ref{T:liminf_Ucells_heat}. When doing so, 
it becomes necessary to stay in the framework of nested fractals as defined in Definition~\ref{def:nested}. 
In particular the assumption
\begin{equation}\label{E:cell_point_intersection}
|K_w \cap K_v| \in \{ 0, 1 \}
\end{equation}
for any $n$-cells $K_w,K_v$ and $n\geq 1$, will play an important role.

\begin{lem}\label{L:asymp_heat_cells}
Let $w,v\in W_n$, $w\neq v$. There exists a periodic function $\theta_{v,w}\colon (0,+\infty) \to \mathbb{R}$  with period $L^{-d_w}$ such that 
\begin{equation}\label{E:asymp_heat_cells}
\mathcal{M}_{K_w,K_v}(t)=  
|K_w \cap K_v|
\, t^{\frac{d_h}{d_w}} \theta_{v,w} (-\ln t) +o \big( t^{d_h/d_w} \big)
\end{equation}
as $t\to 0^+$.
\end{lem}
\begin{proof}
Assume first $K_w \cap K_v=\{ q \}$. Consider the similitude $\psi$ with fixed point $q$ and contraction factor $L^{-1}$. The localization property in Lemma~\ref{L:localization} implies
\[
\mathcal{M}_{K_w,K_v}(t) \simeq \mathcal{M}_{\Psi(K_w),\Psi(K_v)}(t).
\]
Moreover, due to the scaling in property in Lemma~\ref{L:scaling lemma}, we also have
\[
\mathcal{M}_{\Psi(K_w),\Psi(K_v)}(t) \simeq M\mathcal{M}_{K_w,K_v}(L^{-d_w}t).
\]
Therefore,
\[
\mathcal{M}_{K_w,K_v}(t) \simeq M \mathcal{M}_{K_w,K_v}(L^{-d_w}t)
\]
and the renewal Lemma~\ref{L:renewal_lemma} yields~\eqref{E:asymp_heat_cells}. 
If $K_w \cap K_v=\emptyset$, then $d(K_v,K_w) >0$ so that
$\mathcal{M}_{K_w,K_v}(t) \simeq 0$ and the result follows as well.
\end{proof}

Finally we show that the periodic function is independent of the $n$-cells considered.
\begin{lem}\label{L:periodic_indep}
There exists a periodic function $\theta\colon (0,\infty)\to\mathbb{R}$ with period $L^{-d_w}$ such that for any pair of intersecting $n$-cells $K_w,K_v$
\[
\lim_{s \to +\infty }[ \theta_{w,v} (s) -\theta (s)]=0.
\]
\end{lem}

\begin{proof}
Indeed, one has first
\[
\mathcal{M}_{K_w,K_v}(t)  \simeq \int_{K_w}\int_{K_v}p^{K_w \cup K_v}_t(x,y)\,d\mu^{\langle\infty\rangle}(y)\,d\mu^{\langle\infty\rangle}(x).
\]
By scaling invariance and symmetry, 
\[
\int_{K_w}\int_{K_v}p^{K_w \cup K_v}_t(x,y)\,d\mu^{\langle\infty\rangle}(y)\,d\mu^{\langle\infty\rangle}(x) \simeq \int_{K_1}\int_{K_2}p^{K_1 \cup K_2}_t(x,y)\,d\mu^{\langle\infty\rangle}(y)\,d\mu^{\langle\infty\rangle}(x),
\]
where $K_1,K_2$ are two $1$-cells that intersect at one point. 
Thus, $\mathcal{M}_{K_w,K_v}(t) \simeq \mathcal{M}_{K_1,K_2}(t)$ and Lemma~\ref{L:asymp_heat_cells} implies the assertion with $\theta:=\theta_{1,2}$.
\end{proof}

\begin{proof}[Proof of Theorem~\ref{T:liminf_Ucells_heat}]
Recall that $K^{\langle\infty\rangle}$ satisfies~\eqref{E:cell_point_intersection}, e.g. it is an unbounded nested fractal like those based on the Sierpinski gasket or the Vicsek set. Let $U=\bigcup_{i \in I} K_{w_i}$ be a finite (connected) union of $n$-cells.
Then,

\begin{align*}
 &\int_{K^{\langle\infty\rangle}} \int_{K^{\langle\infty\rangle}} |1_U (x) -1_U(y)| p^{K^{\langle\infty\rangle}}_t(x,y) \, d\mu^{\langle\infty\rangle} (x) d\mu^{\langle\infty\rangle}(y) \\
 &=2 \int_{U} \int_{U^c} p^{K^{\langle\infty\rangle}}_t(x,y) \, d\mu^{\langle\infty\rangle} (x) d\mu^{\langle\infty\rangle}(y) \\
 &= 2 \sum_{i \in I} \int_{K_{w_i}} \int_{U^c} p^{K^{\langle\infty\rangle}}_t(x,y) \, d\mu^{\langle\infty\rangle} (x) d\mu^{\langle\infty\rangle}(y) \\
 &\simeq  2 \sum_{i \in I} \int_{K_{w_i}} \int_{K_{w_i}^*} p^{K^{\langle\infty\rangle}}_t(x,y) \, d\mu^{\langle\infty\rangle} (x) d\mu^{\langle\infty\rangle}(y),
\end{align*}
where 
\[
K_{w_i}^*:=\{ K_w, w \in W_n, \colon  K_w \subset \overline{K^{\langle\infty\rangle}\setminus U}, \, d(K_w ,K_{w_i})=0 \}
\]
denotes the set of $n$-cells in $K^{\langle\infty\rangle}{\setminus}U$ that intersect $K_w$ at one point. Figure~\ref{F:SG_outside_cells} shows an example in the Sierpinski gasket.
\begin{figure}[H]
    \centering
\begin{tikzpicture}[scale=1.5]
\tikzstyle{every node}=[draw,circle,fill=black,minimum size=1.5pt, inner sep=0pt]
\draw ($(0:0)$) node () {} --++ ($(0:2/8)$) node () {} --++ ($(120:2/8)$) node () {} --++ ($(240:2/8)$) node () {};
\draw ($(0:2/8)$) node () {} --++ ($(0:2/8)$) node () {} --++ ($(120:2/8)$) node () {} --++ ($(240:2/8)$) node () {};
\draw ($(60:2/8)$) node () {} --++ ($(0:2/8)$) node () {} --++ ($(120:2/8)$) node () {} --++ ($(240:2/8)$) node () {};
\foreach \a in {0,60}{
\draw ($(\a:2/4)$) node () {} --++ ($(0:2/8)$) node () {} --++ ($(120:2/8)$) node () {} --++ ($(240:2/8)$) node () {};
\foreach \b in{0,60}{
\draw ($(\a:2/4)+(\b:2/8)$)node () {} --++ ($(0:2/8)$) node () {} --++ ($(120:2/8)$) node () {} --++ ($(240:2/8)$) node () {};
}
}
\foreach \c in{0,60}{
\draw ($(\c:2/2)$) node () {} --++ ($(0:2/8)$) node () {} --++ ($(120:2/8)$) node () {} --++ ($(240:2/8)$) node () {};
\draw ($(\c:2/2)+(0:2/8)$) node () {} --++ ($(0:2/8)$) node () {} --++ ($(120:2/8)$) node () {} --++ ($(240:2/8)$) node () {};
\draw ($(\c:2/2)+(60:2/8)$) node () {} --++ ($(0:2/8)$) node () {} --++ ($(120:2/8)$) node () {} --++ ($(240:2/8)$) node () {};
\foreach \a in {0,60}{
\draw ($(\c:2/2)+(\a:2/4)$) node () {} --++ ($(0:2/8)$) node () {} --++ ($(120:2/8)$) node () {} --++ ($(240:2/8)$) node () {};
\foreach \b in{0,60}{
\draw ($(\c:2/2)+(\a:2/4)+(\b:2/8)$)node () {} --++ ($(0:2/8)$) node () {} --++ ($(120:2/8)$) node () {} --++ ($(240:2/8)$) node () {};
}
}
}
\filldraw[red,semitransparent] ($(0:0)+(60:1/4)$) --++ ($(60:6/4)$) --++ ($(300:1/4)$) --++ ($(180:1/4)$) --++ ($(300:1/2)$) --++ ($(60:1/4)$) --++ ($(300:1/4)$) --++ ($(180:3/4)$) --++ ($(300:3/4)$) --++ ($(180:1/4)$) --++ ($(60:1/4)$) --++ ($(180:1/2)$) --++ ($(300:1/4)$) --++ ($(180:1/4)$);
\coordinate[label=left:{\small $K_w$}] (K) at ($(0:0)+(90:1/3)$);
\draw[->] (K) --++ ($(0:1/4)$);
\filldraw[blue, semitransparent] ($(0:0)$) --++ ($(60:2/8)$) --++ ($(300:2/8)$) --++ ($(60:2/8)$) --++ ($(300:2/8)$);
\end{tikzpicture}
    \caption{The set $U$ in red. The two blue cells correspond to $K^*_w$.}
    \label{F:SG_outside_cells}
\end{figure}
Noting that $\sum_{i}|K_{w_i}^*|=|\partial U|$, by virtue of Lemma~\ref{L:asymp_heat_cells} and Lemma~\ref{L:periodic_indep} there exists a periodic function $\theta =2 \theta_{1,2}$ such that
\[
\mathcal{M}_{\mathbf{1}_U}(t)= | \partial U| t^{\frac{d_h}{d_w}} \theta (-\ln t) +o \big( t^{\frac{d_h}{d_w}} \big).
\]
Finally, since $1_U \in BV(K^{\langle\infty\rangle})$, it follows from Theorem~\ref{main:BV3} that
\[
\liminf_{t \to 0} t^{-\frac{d_h}{d_w}} \mathcal{M}_{\mathbf{1}_U}(t) >0
\]
and
\[
\limsup_{t \to 0} t^{-\frac{d_h}{d_w}} \mathcal{M}_{\mathbf{1}_U}(t)  <+\infty,
\]
hence $\theta$ is bounded from below and above and the conclusion follows.
\end{proof}
%
%
%
%
%
%
%
%

\section{Korevaar-Schoen functional}
The aim of this section is to show the part of the main result, Theorem~\ref{T:liminf_Ucells_intro}, concerning the Korevaar-Schoen type functional 
\begin{equation*}
\widetilde{\mathcal{M}}_f(r):=\frac{1}{r^{d_h}}\int_{K^{\langle\infty\rangle}}\int_{B(x,r)}|f(x)-f(y)|\,d\mu^{\langle\infty\rangle}(y)\,d\mu^{\langle\infty\rangle}(x),
\end{equation*}
$f\in BV(X)$. 
The functional $\widetilde{\mathcal{M}}_f(r)$ may be regarded as the metric-measure theoretic version of $\mathcal{M}_f$; note that the factor $r^{-d_h}$ is necessary to obtain the correct scaling property.

\begin{thm}\label{T:liminf_Ucells_KS}
Let $K^{\langle\infty\rangle}$ be an unbounded nested fractal. There exists a bounded periodic function $\Psi: (0,\infty)\to [c,d]$ with period $L^{-1}$ and $0 < c  \le d$ such that, for any finite union of $n$-cells $U\subset K^{\langle\infty\rangle}$,
\begin{equation}\label{E:liminf_Ucells_KS}
\liminf_{r\to 0^+} \frac{1}{r^{d_h}}\Psi(-\ln r)\widetilde{\mathcal{M}}_{\mathbf{1}_U}(r)=|\partial U|
\end{equation}
where $|\partial U|$ is the number of points in the boundary of $U$.
\end{thm}

\begin{remark}
We stress that the function $\Psi$ is independent of the set $U$, see Lemma~\ref{L:theta_cells_KS}.
\end{remark}
To show Theorem~\ref{T:liminf_Ucells_KS} we will make use of the auxiliary functional
\begin{equation}\label{E:def_KS_AB}
\widetilde{\mathcal{M}}_{A,B}(r):=
\frac{1}{r^{d_H}}\int_A\int_{B\cap B(x,r)}d\mu(y)\,d\mu(x)
=\frac{1}{r^{d_H}}\int_A\mu^{\langle\infty\rangle}(B\cap B(x,r))\,d\mu^{\langle\infty\rangle}(x),
\end{equation}
where $A,B\subset K^{\langle\infty\rangle}$ and $r>0$.
 
\subsection{Scaling and localization}
We start by obtaining the analogue of the scaling Lemma~\ref{L:scaling lemma}; note that the walk dimension $d_w$ is not visible yet.
\begin{lem}\label{L:scaling_KS}
For any $w\in W_n$, $n\geq 1$, and compact sets $A,B \subset K^{\langle\infty\rangle}$
\begin{equation}\label{E:scaling_M2cells_KS}
\widetilde{\mathcal{M}}_{A,B}(r) = M^n\widetilde{\mathcal{M}}_{\psi_w(A),\psi_w(B)}(L^{-n}r).
\end{equation}
The same holds when $\psi_w$ is replaced by an invariant rotation of itself .
\end{lem}

\begin{proof}
With the change of variables $x=\psi_w(z)$ and since $M=L^{d_H}$ we have
\begin{align*}
\mathcal{M}_{\psi_w(A),\psi_w(B)}(r)&=\frac{1}{r^{d_H}}\int_{\psi(A)}\int_{\psi_w(B)\cap B(x,r)}d\mu^{\langle\infty\rangle}(y)\,d\mu^{\langle\infty\rangle}(x)\\
&=\frac{1}{r^{d_H}}\int_A\int_{\psi_w(B)\cap B(\psi_w(z),r)}M^{-n}d\mu^{\langle\infty\rangle}(y)\,d\mu^{\langle\infty\rangle}(z)\\
&=\frac{1}{r^{d_H}}M^{-n}\int_A\int_{\psi_w(B\cap B(z,L^nr))}d\mu^{\langle\infty\rangle}(y)\,d\mu^{\langle\infty\rangle}(z)\\
&=\frac{1}{r^{d_H}}M^{-2n}\int_A\int_{B\cap B(z,L^nr)}d\mu^{\langle\infty\rangle}(y)\,d\mu^{\langle\infty\rangle}(z)\\
&=M^{-n}\Big(\frac{L^{-n}}{r}\Big)^{d_H}\int_A\int_{B\cap B(z,L^nr)}d\mu^{\langle\infty\rangle}(y)\,d\mu^{\langle\infty\rangle}(z)\\
&=M^{-n}\widetilde{\mathcal{M}}_{A,B}(rL^n).
\end{align*}
\end{proof}

We now move on to proving the analogue of the localization Lemma~\ref{L:localization}.

\begin{lem}[Localization lemma]\label{L:KS_localization}
Let $A,B\subset K^{\langle\infty\rangle}$ be compact and $\tilde A \subset A, \tilde B \subset B$ such that $d(\tilde A,B \setminus \tilde B) >r_0$, $d(A \setminus \tilde A, B \setminus\tilde B) >r_0$, $d(A\setminus \tilde A,\tilde B) >r_0$ for some $r_0>0$. Then,
\[
\widetilde{\mathcal{M}}_{A,B}(r) = \widetilde{\mathcal{M}}_{\tilde A,\tilde B}(r)
\]
holds for any $0<r\leq r_0$.
In particular, if $A,B$ satisfy $d(A,B)>r_0>0$, then 
\begin{equation}\label{E:KS_zero_loc}
\widetilde{\mathcal{M}}_{A,B}(r)=0
\end{equation}
for any $0< r\leq r_0$.
\end{lem}

\begin{proof}
Let $0<r\leq r_0$. Splitting the double integral,
\begin{align*}
\widetilde{\mathcal{M}}_{A,B}(r) &=\frac{1}{r^{d_H}}\int_{\tilde{A}}\int_{\tilde{B}\cap B(x,r)}d\mu^{\langle\infty\rangle}(y)\,d\mu^{\langle\infty\rangle}(x)+\frac{1}{r^{d_H}}\int_{\tilde{A}}\int_{(B{\setminus}\tilde{B})\cap B(x,r)}d\mu^{\langle\infty\rangle}(y)\,d\mu^{\langle\infty\rangle}(x)\\
&+\frac{1}{r^{d_H}}\int_{A{\setminus}\tilde{A}}\int_{\tilde{B}\cap B(x,r)}d\mu^{\langle\infty\rangle}(y)\,d\mu^{\langle\infty\rangle}(x)+\frac{1}{r^{d_H}}\int_{A{\setminus}\tilde{A}}\int_{(B{\setminus}\tilde{B})\cap B(x,r)}d\mu^{\langle\infty\rangle}(y)\,d\mu^{\langle\infty\rangle}(x)
\end{align*}
and the last three terms vanish because $d(\tilde A , B \setminus  \tilde B) >r_0$ implies $B(x,r)\cap (B\setminus\tilde{B})=\emptyset$ for any $x\in \tilde{A}$ whereas $d(A \setminus \tilde A,\tilde B) >r_0$ implies $B(x,r)\cap \tilde{B}=\emptyset$ for any $x\in A\setminus\tilde{A}$.
\end{proof}

\subsection{Finitely ramified nested fractals}
This section again concentrates on the case of unbounded nested fractals like the unbounded Sierpinski gasket, where any two $n$-cells $K_w,K_v$ intersect at most at one point, i.e.
\begin{equation}\label{E:finitely_ramified}
|K_w\cap K_v|\in\{0,1\}.
\end{equation}
It is now that the walk dimension appears when we take into account that the critical exponent is $\alpha_1=\frac{d_h}{d_w}$, c.f.~\cite[Theorem 5.1]{BV3}.

\begin{lem}\label{L:cell_pair_asymp_KS}
For any $n\in\mathbb{N}$ and $v,w\in W_n$ there exists a periodic function $\theta_{v,w}\colon(0,\infty)\to\mathbb{R}$ such that
\[
\widetilde{\mathcal{M}}_{K_v,K_w}(r)=r^{\alpha_1d_w}\theta_{v,w}(-\ln r)+o(r^{d_H})
\]
as $r\to 0^+$.
\end{lem}

\begin{proof}
Assume first $K_v\cap K_w=\{q\}$. Let $\psi$ denote the similitude having $q$ as fixed point. Choosing $0<r<L^{-1}$ we have that the sets $K_v,\psi(K_v)$ and $K_w,\psi(K_w)$ satisfy the conditions of Lemma~\ref{L:KS_localization} with $A=K_v$, $\tilde{A}=\psi(K_v)$, $B=K_w$ and $\tilde{B}=\psi(K_w)$. Together with the scaling Lemma~\ref{L:scaling_KS} this implies
\[
\widetilde{\mathcal{M}}_{K_v,K_w}(r)=\widetilde{\mathcal{M}}_{\psi(K_v),\psi(K_w)}(r)=M^{-1}\widetilde{\mathcal{M}}_{K_v,K_w}(rL).
\]
In the case that $K_v\cap K_w=\emptyset$, the cells are separated at least by a factor $L^{-1}$, whence~\eqref{E:KS_zero_loc} implies $\widetilde{\mathcal{M}}_{K_v,K_w}(r)=0$. Applying the renewal Lemma~\ref{L:renewal_lemma} yields
\begin{align*}
\widetilde{\mathcal{M}}_{K_v,K_w}(r)&=r^{-\frac{\log M}{\log L^{-1}}}\theta_{v,w}(-\ln r)+o(r^{-\frac{\log M}{\log L^{-1}}})\\
&=r^{\frac{\log M}{\log L}}\theta_{v,w}(-\ln r)+o(r^{\frac{\log M}{\log L}})\\
&=r^{d_h}\theta_{v,w}(-\ln r)+o(r^{d_h})=r^{\frac{d_h}{d_w}d_w}\theta_{v,w}(-\ln r)+o(r^{d_h}).
\end{align*}
\end{proof}
As in the case treated in Section~\ref{S:HS_functional}, a consequence of the translation and rotation invariance of the functional $\widetilde{\mathcal{M}}_{K_v,K_w}(r)$ is that the periodic function appearing in Lemma~\ref{L:cell_pair_asymp_KS} is independent of the pair of cells $K_v,K_w$.

\begin{lem}\label{L:theta_cells_KS}
There exists a periodic function $\theta\colon (0,\infty)\to\mathbb{R}$ with period $L^{-1}$ such that for any $n\in\mathbb{N}$ and any pair of intersecting cells $K_v,K_w\subset K^{\langle\infty\rangle}$ with $v,w\in W_n$,
\[
\liminf_{s\to \infty}[\theta_{v,w}(s)-\theta(s)]=0.
\]
\end{lem}

\begin{proof}
Let $K_1,K_2\subset K^{\langle\infty\rangle}$ be two $1$-cells that intersect at one point. Note that $K_2\cap B(x,r)$ is empty for any $x\in K_1$ with $0<r<L^{-n}$ and thus translation and rotation invariance imply
\begin{equation*}\label{E:theta_cells_KS_01}
\widetilde{\mathcal{M}}_{K_v,K_w}(r)=\frac{1}{r^{d_h}}\int_{K_v}\int_{K_w\cap B(x,r)}d\mu(y)\,d\mu(x)
=\frac{1}{r^{d_h}}\int_{K_1}\int_{K_2\cap B(x,r)}d\mu(y)\,d\mu(x)=\widetilde{\mathcal{M}}_{K_1,K_2}(r)
\end{equation*}
for $r>0$ sufficiently small. By virtue of Lemma~\ref{L:cell_pair_asymp_KS} there exist periodic functions $\theta_{v,w}(s)$ and $\theta_{1,2}(s)$ such that
\begin{align*}
\liminf_{s\to\infty}\big(\theta_{v,w}(s)-\theta_{1,2}(s)\big)
&=\liminf_{r\to 0^+}\theta_{v,w}(-\ln r)-\theta_{1,2}(-\ln r)\big)\\
&=\liminf_{r\to 0^+}\big(r^{-d_h}\widetilde{\mathcal{M}}_{K_v,K_w}(r)-r^{-d_h}\widetilde{\mathcal{M}}_{K_1,K_2}(r)\big)=0,
\end{align*}
which gives the assertion with $\theta:=\theta_{1,2}$.
\end{proof}

\begin{proof}[Proof of Theorem~\ref{T:liminf_Ucells_KS}]
Write $U=\bigcup_{i=1}^N K_{w_i}$, $N>0$, where $K_{w_i}\subset K^{\langle\infty\rangle}$ are $n$-cells. 
As in Theorem~\ref{T:liminf_Ucells_heat} define for each $i=1,\ldots, N$ the index set
\[
J_{w_i}:=\{j=1,\ldots,N\colon K_{w_j}\subset\overline{K^{\langle\infty\rangle}\setminus U},~ d(K_{w_i},K_{w_j})=0\}
\]
and note that $d(K_{w_i},K_{w_j})=0$ is equivalent to $d(K_{w_i},K_{w_j})<r$ for any $0<r<{\rm diam}\, K_{w_i}/2$. Thus, for each such $r$ we obtain
\begin{align*}
\mathcal{M}_{\mathbf{1}_U}(r)&=\frac{1}{r^{d_h}}\int_{K^{\langle\infty\rangle}}\int_{B(x,r)\cap K^{\langle\infty\rangle}}|\mathbf{1}_U(x)-\mathbf{1}_U(y)|\,d\mu(y)\,d\mu(x)\\
&=\frac{1}{r^{d_h}}\int_U\int_{B(x,r)\cap U^c}d\mu(y)\,d\mu(x)+\frac{1}{r^{d_h}}\int_{U^c}\int_{B(x,r)\cap U}d\mu(y)\,d\mu(x)\\
&=\sum_{i=1}^N\frac{1}{r^{d_h}}\int_{K_{wi}}\int_{B(x,r)\cap U^c}d\mu(y)\,d\mu(x)+\sum_{i=1}^N\frac{1}{r^{d_h}}\int_{U^c}\int_{B(x,r)\cap K_{wi}}d\mu(y)\,d\mu(x)\\
&=\sum_{i=1}^N\sum_{j\in J_{w_i}}\frac{1}{r^{d_h}}\int_{K_{wi}}\int_{B(x,r)\cap K_{w_j}}d\mu(y)\,d\mu(x)+\sum_{i=1}^N\sum_{j\in J_{w_i}}\frac{1}{r^{d_h}}\int_{K_{wj}}\int_{B(x,r)\cap K_{w_i}}d\mu(y)\,d\mu(x)\\
&=\sum_{i=1}^N\sum_{j\in J_{w_i}}\Big(\mathcal{M}_{K_{w_i},K_{w_j}}(r)+\mathcal{M}_{K_{w_j},K_{w_i}}(r)\Big).
\end{align*}
By symmetry, Lemma~\ref{L:cell_pair_asymp_KS} and Lemma~\ref{L:theta_cells_KS} implies that, as $r\to 0^+$, the latter sum equals
\[
2\sum_{i=1}^N|J_{w_i}|(r^{-\alpha_1 d_w}\theta_{1,2}(-\ln r)+o(r^{d_h})=2|\partial U|\big(r^{-\alpha_1 d_w}\theta_{1,2}(-\ln r)+o(r^{d_h})\big),
\]
hence~\eqref{E:liminf_Ucells_KS} holds with $\Psi(z)=(2\theta_{1,2}(z))^{-1}$.

\medskip

To justify that the $\liminf$ is non-zero we invoke the definition of variation introduced in~\cite[Section 4.2]{BV3} that is given by
\[
{\rm Var}(f)=\liminf_{r\to 0^+}\frac{1}{r^{\alpha_1d_w}}\widetilde{\mathcal{M}}_f(r),
\]
where $\alpha_1>0$ is the critical exponent. From~\cite[Theorem 5.1]{BV3} it follows that $\alpha_1d_w=d_h$ for p.c.f. fractals, which together with~\cite[Theorem 4.9]{BV3} yields
\[
\liminf_{r\to 0^+} \frac{1}{r^{d_h}}\widetilde{\mathcal{M}}_{\mathbf{1}_U}(r)={\rm Var}(\mathbf{1}_U)\geq C\|\mathbf{1}_{U}\|_{1,d_h/d_w}.
\]
Since $\mathbf{1}_{U}$ is a non-constant function, the seminorm above is non-zero and hence~\eqref{E:liminf_Ucells_KS} is also non-zero as long as $\Psi$ is non-zero.
\end{proof}

\subsection{Non-existence of the limit}\label{SS:KS_non-limit}
The geometric nature of the functional $\widetilde{\mathcal{M}}_{A,B}$ defined in~\eqref{E:def_KS_AB} makes it possible to prove the non-trivial oscillations of the function $\Psi$ in~\eqref{E:liminf_Ucells_KS}. In this section we continue working with p.c.f. nested fractals and provide explicit details for the case of the Sierpinski gasket and the Vicsek set.

\medskip

First, note that the proof of Theorem~\ref{T:liminf_Ucells_KS} indicates that it suffices to study the (non)convergence of $\widetilde{\mathcal{M}}_{K_u,K_v}(r)$ for any pair of $n$-cells that meet at a point $p\in V^{(n)}$. Because the latter functional can be written as
\begin{equation}\label{E:KS_as_measure}
\widetilde{\mathcal{M}}_{K_u,K_v}(r)=\frac{1}{r^{d_H}}\mu\otimes\mu\big(\{(x,y)\in K_u\times K_v\colon d(x,y)\leq r\}\big),
\end{equation}
the main idea consists in approximating that quantity for different sequences $\{r_m\}_{m\geq 1}$ $r_m\to 0^+$. The choice of the sequences in the following lemma is based in the observation illustrated in Figure~\ref{F:SG_VS_sandwich}: 

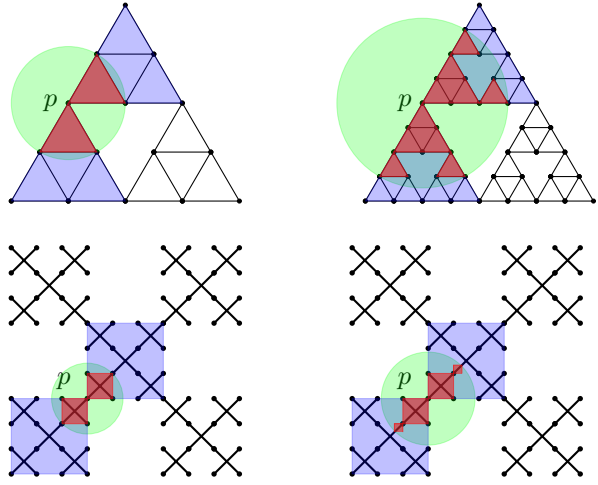
\begin{figure}[H]
\centering
\begin{tabular}{cc}
\begin{tikzpicture}[scale=1.5]
\tikzstyle{every node}=[draw,circle,fill=black,minimum size=1.5pt, inner sep=0pt]
\draw ($(0:0)$) node () {} --++ ($(0:2/4)$) node () {} --++ ($(120:2/4)$) node () {} --++ ($(240:2/4)$) node () {};
\draw ($(0:2/4)$) node () {} --++ ($(0:2/4)$) node () {} --++ ($(120:2/4)$) node () {} --++ ($(240:2/4)$) node () {};
\draw ($(60:2/4)$) node () {} --++ ($(0:2/4)$) node () {} --++ ($(120:2/4)$) node () {} --++ ($(240:2/4)$) node () {};
\foreach \a in {0,60}{
\draw ($(\a:2/2)$) node () {} --++ ($(0:2/4)$) node () {} --++ ($(120:2/4)$) node () {} --++ ($(240:2/4)$) node () {};
\foreach \b in{0,60}{
\draw ($(\a:2/2)+(\b:2/4)$)node () {} --++ ($(0:2/4)$) node () {} --++ ($(120:2/4)$) node () {} --++ ($(240:2/4)$) node () {};
}
}
\coordinate[label=left:{$p\;$}] (p) at ($(0:0)+(60:1)$);
\filldraw[green,nearly transparent] (p) circle (.5 cm);
\filldraw[blue,nearly transparent] ($(0:0)$) --++ ($(60:2)$) --++ ($(300:1)$) --++ ($(180:1)$) --++ ($(300:1)$) --++ ($(180:1)$);
\filldraw[red,semitransparent] ($(0:0)+(60:1/2)$) --++ ($(60:1)$) --++ ($(300:1/2)$) --++ ($(180:1/2)$) --++ ($(300:1/2)$) --++ ($(180:1/2)$);
\end{tikzpicture}
\hspace*{2em}&
\begin{tikzpicture}[scale=1.5]
\tikzstyle{every node}=[draw,circle,fill=black,minimum size=1.5pt, inner sep=0pt]
\draw ($(0:0)$) node () {} --++ ($(0:2/8)$) node () {} --++ ($(120:2/8)$) node () {} --++ ($(240:2/8)$) node () {};
\draw ($(0:2/8)$) node () {} --++ ($(0:2/8)$) node () {} --++ ($(120:2/8)$) node () {} --++ ($(240:2/8)$) node () {};
\draw ($(60:2/8)$) node () {} --++ ($(0:2/8)$) node () {} --++ ($(120:2/8)$) node () {} --++ ($(240:2/8)$) node () {};
\foreach \a in {0,60}{
\draw ($(\a:2/4)$) node () {} --++ ($(0:2/8)$) node () {} --++ ($(120:2/8)$) node () {} --++ ($(240:2/8)$) node () {};
\foreach \b in{0,60}{
\draw ($(\a:2/4)+(\b:2/8)$)node () {} --++ ($(0:2/8)$) node () {} --++ ($(120:2/8)$) node () {} --++ ($(240:2/8)$) node () {};
}
}
\foreach \c in{0,60}{
\draw ($(\c:2/2)$) node () {} --++ ($(0:2/8)$) node () {} --++ ($(120:2/8)$) node () {} --++ ($(240:2/8)$) node () {};
\draw ($(\c:2/2)+(0:2/8)$) node () {} --++ ($(0:2/8)$) node () {} --++ ($(120:2/8)$) node () {} --++ ($(240:2/8)$) node () {};
\draw ($(\c:2/2)+(60:2/8)$) node () {} --++ ($(0:2/8)$) node () {} --++ ($(120:2/8)$) node () {} --++ ($(240:2/8)$) node () {};
\foreach \a in {0,60}{
\draw ($(\c:2/2)+(\a:2/4)$) node () {} --++ ($(0:2/8)$) node () {} --++ ($(120:2/8)$) node () {} --++ ($(240:2/8)$) node () {};
\foreach \b in{0,60}{
\draw ($(\c:2/2)+(\a:2/4)+(\b:2/8)$)node () {} --++ ($(0:2/8)$) node () {} --++ ($(120:2/8)$) node () {} --++ ($(240:2/8)$) node () {};
}
}
}
\coordinate[label=left:{$p\;$}] (p) at ($(0:0)+(60:1)$);
\filldraw[green,nearly transparent] (p) circle (.75 cm);
\filldraw[blue,nearly transparent] ($(0:0)$) --++ ($(60:2)$) --++ ($(300:1)$) --++ ($(180:1)$) --++ ($(300:1)$) --++ ($(180:1)$);
\filldraw[red,semitransparent] ($(0:0)+(60:1/4)$) --++ ($(60:6/4)$) --++ ($(300:1/4)$) --++ ($(180:1/4)$) --++ ($(300:1/2)$) --++ ($(60:1/4)$) --++ ($(300:1/4)$) --++ ($(180:3/4)$) --++ ($(300:3/4)$) --++ ($(180:1/4)$) --++ ($(60:1/4)$) --++ ($(180:1/2)$) --++ ($(300:1/4)$) --++ ($(180:1/4)$);
\end{tikzpicture}
\\
&\\
\begin{tikzpicture}[scale=1/3]
\foreach \c/\d in {0/0, 0/6, 3/3, 6/0, 6/6}{
\foreach \a/\b in {0/0, 0/2, 1/1, 2/0, 2/2}{
\fill[shift={(\c,\d)}] (\a,\b) circle (3pt);
\fill[shift={(\c,\d)}] (\a,\b+1) circle (3pt);
\fill[shift={(\c,\d)}] (\a+1,\b+1) circle (3pt);
\fill[shift={(\c,\d)}] (\a+1,\b) circle (3pt);
\draw[thick,shift={(\c,\d)}] (\a,\b) -- (\a+1,\b+1) (\a,\b+1) -- (\a+1,\b);
}
}
\coordinate[color=red, label=above left :{$p\;$}] (p) at (3,3);
\filldraw[green,nearly transparent] (p) circle (1.41 cm);
\filldraw[blue,nearly transparent] ($(0:0)$) --++ ($(90:3)$) --++ ($(0:3)$) --++ ($(90:3)$) --++ ($(0:3)$) --++ ($(270:3)$) --++ ($(180:3)$) --++ ($(270:3)$) --++ ($(180:3)$);
\filldraw[red,semitransparent] (2,2) --++ ($(90:1)$) --++ ($(0:1)$) --++ ($(90:1)$) --++ ($(0:1)$) --++ ($(270:1)$) --++ ($(180:1)$) --++ ($(270:1)$) --++ ($(180:1)$);
\end{tikzpicture}
\hspace*{2em}
&
\begin{tikzpicture}[scale=1/3]
\foreach \c/\d in {0/0, 0/6, 3/3, 6/0, 6/6}{
\foreach \a/\b in {0/0, 0/2, 1/1, 2/0, 2/2}{
\fill[shift={(\c,\d)}] (\a,\b) circle (3pt);
\fill[shift={(\c,\d)}] (\a,\b+1) circle (3pt);
\fill[shift={(\c,\d)}] (\a+1,\b+1) circle (3pt);
\fill[shift={(\c,\d)}] (\a+1,\b) circle (3pt);
\draw[thick,shift={(\c,\d)}] (\a,\b) -- (\a+1,\b+1) (\a,\b+1) -- (\a+1,\b);
}
}
\coordinate[label=above left:{$p\;$}] (p) at (3,3);
\filldraw[green,nearly transparent] (p) circle (1.85 cm);
\filldraw[blue,nearly transparent] ($(0:0)$) --++ ($(90:3)$) --++ ($(0:3)$) --++ ($(90:3)$) --++ ($(0:3)$) --++ ($(270:3)$) --++ ($(180:3)$) --++ ($(270:3)$) --++ ($(180:3)$);
\filldraw[red,semitransparent] (2,2) --++ ($(90:1)$) --++ ($(0:1)$) --++ ($(90:1)$) --++ ($(0:1)$) --++ ($(270:1)$) --++ ($(180:1)$) --++ ($(270:1)$) --++ ($(180:1)$);
\filldraw[red,semitransparent] (5/3,5/3) --++ ($(90:1/3)$) --++ ($(0:1/3)$) --++ ($(270:1/3)$);
\filldraw[red,semitransparent] (4,4) --++ ($(90:1/3)$) --++ ($(0:1/3)$) --++ ($(270:1/3)$);
\end{tikzpicture}
\end{tabular}
\caption{Intersections of $n$-cells with balls of radius $r_n$ (left) and $r'_n$ (right) in the Sierpinski gasket (above) and the Vicsek set (below).}
\label{F:SG_VS_sandwich}
\end{figure}
On the one hand, if $p\in V^{(n)}$ is the vertex of an $n$-cell $K_{w}$ and $r_n:=({\rm diam}\, K)L^{-n}$, 
\begin{equation}
    B(p,r)\cap K_w=K_w\qquad\forall r> r_n
\end{equation}
while 
\begin{equation}
    B(p,r)\cap K_w\varsubsetneq K_w\qquad\forall r< r_n.
\end{equation}
On the other hand, a ball $B(p,r')$ with radius $r'>r'_n:=({\rm diam}\, K)L^{-n}+({\rm diam}\, K)L^{-n-1}$ will also cover those $(n+1)$-cells in the $(n-1)$-cell that contained $K_w$ but did not belong to $K_w$ itself. That is
\begin{equation}
    B(p,r)\cap K_w=K_w\cup K_w^*\qquad\forall r> r'_n 
\end{equation}
where
\begin{equation}
  K_w^*:=\bigcup_{\substack{\tilde{w}\in W_{n+1}\\ K_{\tilde{w}}\cap K_w^c\neq\emptyset}}K_{\tilde{w}},
\end{equation}
and 
\begin{equation}
    B(p,r)\cap K_w\subset K_w\cup K_w^* \qquad\forall r< r'_n.
\end{equation}
Finally, notice that the number of $(n+1)$-cells in $K^*_w$ is independent of the cell and the level. Using the previous notation, the non-existence of the limit~\eqref{E:liminf_Ucells_KS} follows from the next lemma.
\begin{lem}\label{L:KS_cells_SG}
Let $(K_u,K_v)$ denote a pair of $n$-cells with $K_u\cap K_v=\{p\}$ and let
\begin{equation}
  R:=\#\{ i\in W_n\colon  i\neq 1, K_i\cap K_1^c\neq\emptyset\}.
\end{equation}
For any $m> n$ large enough, 
\begin{equation}\label{E:KS_sandwich_SG}
\mu\otimes\mu\big(\{(x,y)\in K_u\times K_v\colon d(x,y)\leq r_m\}\big)=
\begin{cases}
2M^{-m}&\text{for }r_m=({\rm diam}\, K)L^{-m},\\
2M^{-m}(1+RM^{-1})&\text{for }r_m=({\rm diam}\, K)L^{-m}(1+L^{-1}).
\end{cases}
\end{equation}
\end{lem}

\begin{proof}
Let $p\in V_n\setminus V_{n-1}$ be the vertex where the pair of $n$-cells $(K_u,K_v)$ intersect. By construction, see also Figure~\ref{F:SG_VS_sandwich}, for $m>n$ sufficiently large we have
\begin{equation}
\big(K_{u}\cup K_{v}\big)\cap B(p,r_m)=
\begin{cases}
K_{wu}\cup K_{wv}\cup \{p\}&\text{for }r_m=({\rm diam}\, K)L^{-m},\\
K_{wu}^*\cup K_{wv}^*\cup \{p\}&\text{for }r_m=({\rm diam}\, K)L^{-m}(1+L^{-1}),
\end{cases}
\end{equation}
where $K_{wu}$ and $K_{wv}$, $w\in W_{m-n+1}$ denote the $(m+1)$-cells that intersect at $p$, and $K_{wu}^*$ and $K_{wv}^*$ their corresponding (outer) $L^{-(m+1)}$-neighborhoods.

On the one hand, the pairs of cells $(K_{wu},K_{wv})$ only intersect at one point, hence
\begin{align*}
&\mu\otimes\mu\big(\{(x,y)\in K_{wu}\times K_{wv}\colon d(x,y)\leq L^{-m}\} \big)\\
&= \mu\big((K_{wu}\cup K_{wv})\cap B(p,L^{-m})\big)
=\mu(K_{wu}\cup K_{wv})=2\mu(K_{wu})=2M^{-{m}}.
\end{align*}
On the other hand,
\begin{align*}
&\mu\otimes\mu\big(\{(x,y)\in K_{wu}\times K_{wv}\colon d(x,y)\leq L^{-m}(1+L^{-1})\} \big)\\
&= \mu\big((K_{wu}^*\cup K_{wv}^*)\cap B(p,L^{-m}(1+L^{-1})\big)
=\mu(K_{wu}^*\cup K_{wv}^*)=2\mu(K_{wu}^*)=2M^{-{m}}(1+RM^{-1}).
\end{align*}
\end{proof}
Since $M=L^{d_h}$, in view of~\eqref{E:KS_as_measure} and Lemma~\ref{L:KS_cells_SG} we can compute explicitly the two limits and confirm that they are different.

\begin{cor}\label{C:KS_liminf_SG}
For any pair of $n$-cells $(K_u,K_v)$ in a p.c.f. nested fractal, 
\begin{equation}\label{E:KS_liminf_SG}
\lim_{m\to\infty}\widetilde{\mathcal{M}}_{K_u,K_v}(r_m)=
\begin{cases}
2&\text{for }r_m=({\rm diam}\, K)^{-m},\\
2(1+2M^{-1})&\text{for }r_m=({\rm diam}\, K)L^{-m}(1+L^{-1}).
\end{cases}
\end{equation}
Consequently, the limit~\eqref{E:liminf_Ucells_KS} does not exist.
\end{cor}
In particular we have $R=2$ in the Sierpinski gasket and $R=1$ in the Vicsek set.
\bibliographystyle{amsplain}
\bibliography{BVfractals_Refs}

\vspace*{2em}

\noindent
\textbf{Fabrice Baudoin}: \url{fabrice.baudoin@uconn.edu}\\
Department of Mathematics,
University of Connecticut,
USA

\medskip

\noindent
\textbf{Patricia Alonso Ruiz}: \url{paruiz@tamu.edu}\\
Department of Mathematics
Texas A\&M University
USA

\end{document}